\documentclass[a4paper,12pt]{amsart}
\usepackage[english]{babel}
\usepackage[utf8]{inputenc}
\usepackage{amssymb}
\usepackage{amsmath}
\usepackage{amsfonts}
\usepackage{amstext}
\usepackage{amsthm}
\usepackage{todonotes}
\usepackage{a4wide}

\usepackage{enumerate}
\usepackage{hyperref}

\newcommand{\IN}{{\mathbb{N}}}
\newcommand{\IR}{{\mathbb{R}}}

\newcommand{\IP}{{\mathbb{P}}}
\newcommand{\IE}{{\mathbb{E}}}

\newcommand{\1}{\mathbbmss{1}}

\renewcommand{\phi}{\varphi}
\renewcommand{\epsilon}{\varepsilon}

\newcommand{\A}{\mathcal{A}}

\newcommand{\ltwo}{{\ell^2(X,m)}}

\newtheorem{theorem}{Theorem}[section]
\newtheorem{lemma}[theorem]{Lemma}
\newtheorem{proposition}[theorem]{Proposition}

\newtheorem{remarks}[theorem]{Remark}
\newtheorem{definition}[theorem]{Definition}

\newtheorem{assumption}[theorem]{Assumption}

\begin{document}

\title{A note on Neumann problems on graphs}

\author[Hinz]{Michael Hinz$^1$}
\address{M. Hinz,  Fakult\"at f\"ur Mathematik \\ Universit\"at Bielefeld \\
 33501 Bielefeld, Germany} \email{mhinz@math.uni-bielefeld.de}

\author[Schwarz]{Michael Schwarz}
\address{M. Schwarz, Institut f{\"u}r Mathematik \\Universit{\"a}t Potsdam \\14476 Potsdam, Germany}\email{mschwarz@math.uni-potsdam.de}

\thanks{$(^1)$ Research supported in part by the DFG IRTG 2235: 'Searching for the regular in the irregular: Analysis of singular and random systems'.}

\begin{abstract}
We discuss Neumann problems for self-adjoint Laplacians on (possibly infinite) graphs. Under the assumption
that the heat semigroup is ultracontractive we discuss the unique solvability for non-empty
subgraphs with respect to the vertex boundary and provide analytic and probabilistic representations for 
Neumann solutions. A second result deals with Neumann problems on canonically compactifiable graphs with respect to
the Royden boundary and provides conditions for unique solvability and analytic and probabilistic representations.
\end{abstract}
\maketitle

\section{Introduction}

The present note discusses aspects of Neumann boundary value problems for self-adjoint Laplace operators 
on graphs in the framework of \cite{KL10, KL12}, which in particular can accommodate finite and infinite weighted graphs. Our goals are to formulate Neumann problems in this framework, to provide sufficient conditions for their unique solvability and to implement analytic and probabilistic representations for their solutions. The first part of the note is rather expository. There we consider Neumann problems for subgraphs and their vertex boundaries. The second part deals with Neumann problems for self-adjoint Laplacians on canonically compactifiable graphs, then with respect to the Royden boundary, an ideal boundary particularly suitable to discuss energy finite functions, \cite{CC63, Doob62, kasue, KLSW17, Maeda64, Soardi} see also \cite{H13}. This can be seen as a counterpart to the study of Dirichlet problems with respect to the Royden boundary, \cite{Soardi, KLSW17}, and complements earlier results by Kasue, \cite{kasue}.

The literature on discrete Dirichlet problems (with respect to the vertex boundary) is quite abundant, see for instance \cite{Chung97, Woess} for finite and \cite{KLW} for possibly infinite graphs. Probabilistic representations of solution to Dirichlet problems (see e.g. \cite[Chapter II, (1.15) Theorem]{Bass95}) with respect to the vertex boundary are well known for Laplacians on graphs, see for instance \cite{Woess}. For Euclidean domains they are originally due to Doob, \cite{Doob54}. There is a rich literature on Dirichlet problems on infinite graphs with respect to abstract ideal boundaries, we mention only \cite{Doob59, Woess00, Woess} and in particular, \cite{KLSW17, Soardi} for the case of the Royden boundary.

Fewer references discuss Neumann Laplacians on graphs. From operator and spectral theoretic perspectives they have been investigated in \cite{ChY94, ChGY96, ChY97, Tan99} for finite graphs and in \cite{HKLW, HKMW} for general graphs.
A study of Neumann problems on finite graphs with respect to the vertex boundary, including a probabilistic formula for solutions, has been provided in \cite{BenMen05}, see in particular their Theorem 8.1. A probabilistic formula for the solution of the Neumann problem in Euclidean domains is due to Brosamler, \cite{brosamler}, see also \cite{Hsu85} for Schr\"odinger operators, \cite{bass_hsu} for Lipschitz domains, and \cite{BM09} for more general differential operators. 
In \cite{kasue} Neumann problems on transient graphs with respect to the Kuramochi boundary were studied, and the
existence and uniqueness of solutions was verified, \cite[Theorem 7.9]{kasue}. What we could not find in the existing literature are more concrete representation formulas for Neumann solutions on subgraphs or with respect to the abstract Royden boundary. 

In Section ~\ref{section:Brosamler_finite} of the present paper we use the framework of \cite{KL10, KL12} and follow the methods of \cite{bass_hsu} to show existence and uniqueness of solutions to the Neumann problem for subgraphs with respect to the vertex boundary under the assumption of ultracontractivity of the associated semigroup and to establish analytic and probabilistic representations, Theorem~\ref{theoremvertexboundary}. A main tool is the uniform exponential convergence to equilibrium, \cite[Theorem 2.4]{bass_hsu}, see also \cite[Proposition 2.2]{Pardoux99}. 

In Section~\ref{section:brosamler_royden} we consider canonically compactifiable graphs and their Royden boundaries. These graphs were introduced and studied in \cite{canon}, and recent research, c.f. \cite{canon, KS17}, indicates that they may be seen as discrete analogues of smooth bounded Euclidean domains. Their Royden compactifications are metrizable, as can be verified by elementary arguments, Lemma \ref{lemma:royden_metrizable}. Since Royden compactifications are compact Hausdorff spaces, metrizability is equivalent
to second countability, and the latter is a crucial ingredient for the construction of a sufficiently well behaved Markov process on the whole compacification, see \cite[Chapter I, (9.4) Theorem]{BG68}, needed for probabilistic representations of Neumann solutions. Moreover, the semigroup of the Neumann Laplacian on a canonically compactifiable graph is ultracontractive, so that we may again follow \cite{bass_hsu} to establish unique solvability and analytic and probabilistic representations, see Theorem~\ref{theorem:solution_royden}. Although for canonically compactifiable graphs the Royden boundary coincides with the Kuramochi boundary (this is shortly discussed in \cite[Section 3]{K17}) our study differs from that in \cite[Section 7]{kasue}. While this reference uses abstract techniques to establish existence and uniqueness, our study relies on the use of heat kernels and, different to \cite[Section 7]{kasue}, provides concrete representations of the solutions. In this sense our results may be seen as complements to \cite{kasue}. 

Both Section~\ref{section:Brosamler_finite} and Section~\ref{section:brosamler_royden} follow the same schedule: We introduce a Dirichlet form on the space including the boundary and define normal derivatives on the boundary via a Gauss-Green type formula as usual. We formulate the Neumann problem, show  that solutions for the same boundary data differ only by constants and that a solution must have centered boundary data and give a sufficient condition for a function to be a Neumann solution. Using variants of the Revuz correspondence (and assuming ultracontractivity in the case of the vertex boundary) we then prove Theorems ~\ref{theoremvertexboundary} and \ref{theorem:solution_royden}.

In Section~\ref{section:Preliminaries} we introduce graphs and some related objects, such as Laplacians and Neumann forms. Some material on Feller transition functions, Dirichlet forms, semigroups, heat kernels and the Royden compactification of canonically compactifiable graphs are collected in an appendix. 

\section*{Acknowledgements}

The authors would like to thank Matthias Keller for helpful discussions and for pointing out various references.

\section{Preliminaries}\label{section:Preliminaries}

Following the presentations in \cite{KL10, KL12} we introduce graphs and some associated objects needed to define normal derivatives and to phrase the Neumann problem.

Let $X$ be an at most countable set. A \emph{graph} on $X$ is a symmetric function $b:X\times X\to [0,\infty)$ that vanishes on the diagonal and satisfies $\sum_{y\in X} b(x,y)<+\infty$
for every $x\in X$. A graph is called \emph{connected} if for every $x,y\in X$ there are $x_1,\ldots, x_n\in X$ such that $b(x,x_1)>0, b(x_1,x_2)>0,\ldots, b(x_n,y)>0$. We equip $X$ with the discrete topology, so that all real valued functions on $X$ are continuous, and we denote their vector space by $C(X)$. By 
\begin{equation}\label{E:Q}
\widetilde{Q}(u)=\frac12\sum_{x,y\in X} b(x,y)(u(x)-u(y))^2
\end{equation}
we define a quadratic form $\widetilde{Q}:C(X)\to [0,\infty]$, its effective domain is denoted by $\widetilde{D}:=\{u\in C(X)\colon \widetilde{Q}(u)<+\infty\}$. On $\widetilde{D}$ the form $\widetilde{Q}$ induces a symmetric non-negative definite bilinear form via polarization, also denoted by $\widetilde{Q}$. 

Every Borel measure $m$ on $X$ is given by a function $m:X\to [0,\infty]$, via
$m(A):=\sum_{x\in A} m(x)$. The space of real valued functions on $X$ which are square-summable with respect to $m$ is denoted by $\ltwo$, with norm $\|\cdot\|_{\ell^2(X,m)}$ and scalar product $\langle \cdot,\cdot\rangle_{\ell^2(X,m)}$. If $m(x)>0$ for all $x\in X$ then we call $(X,m)$ a \emph{discrete measure space} and say it is \emph{finite} if $m(X)<+\infty$. 

Suppose $(X,m)$ is a discrete measure space. Since by Fatou's lemma the form $\widetilde{Q}$ is lower semicontinuous with respect to pointwise convergence, the restriction $Q$ of $\widetilde{Q}$ to 
\[D(Q):=\widetilde{D}\cap\ltwo\]
is a Dirichlet form. For details on Dirichlet forms see Appendix \ref{appendix:Dirichlet} and \cite{FOT, KLW}. Its generator $L$, (\ref{E:saop}), turns out to be a restriction of the \emph{formal Laplacian} $\widetilde{L}$, defined on the space $\widetilde{F}$ consisting of all $f\in C(X)$ with $\sum_{y\in X} b(x,y)|f(x)|<+\infty$ for all $x\in X$ 
by 
\[\widetilde{L}f(x)=\frac{1}{m(x)}\sum_{y\in X} b(x,y)(u(x)-u(y)),\] 
$x\in X$, c.f. \cite{HKLW}. The form $Q$ is called the \emph{Neumann form for $b$} and the operator $L$ is called the \emph{Neumann Laplacian for $b$}. 

\section{Neumann problems for the vertex boundary}\label{section:Brosamler_finite}

W discuss Neumann problems on the vertex boundary of (sub-)graphs. Let $b$ be a graph over a set $X$ and $A\subsetneq X$ be non-empty. We define the \emph{vertex boundary} of $A$ as 
$$\partial_v A:=\{y\in X\setminus A:\text{ there is }x\in A\text{ such that }b(x,y)>0\}$$
and set $\bar{A}=A\cup \partial_v A.$ On $\bar{A}$ we induce a graph structure $b_{\bar{A}}$ by defining $b_{\bar{A}}(x,y)=b(x,y)$ for $x,y\in A$, or $x\in A,y\in \partial_v A$, or $x\in\partial_v A,y\in A$, and $b_{\bar{A}}(x,y)=0$ otherwise. We make the following standing assumption.

\begin{assumption}\label{A:basicdiscrete}
The graph $b_{\bar{A}}$ on $\bar{A}$ is connected, and $\bar{A}$ is equipped with a measure $m$ such that $(\bar{A},m)$ is a finite discrete measure space. 
\end{assumption}

Denote the Neumann form for $b_{\bar{A}}$, defined as in (\ref{E:Q}), by $Q_{\bar{A}}$ and write $L_{\bar{A}}$ for the Neumann Laplacian. Since $m$ is finite, we have $\mathbf{1}_{\bar{A}}\in D(Q_{\bar{A}})$.
The operator $L_{\bar{A}}$ is a restriction of the formal Laplacian $\widetilde{L}_{\bar{A}}$ for the graph $b_{\bar{A}}$, given by      
\[\widetilde{L}_{\bar{A}}f(x)=\frac{1}{m(x)}\sum_{y\in \bar{A}}b_{\bar{A}}(x,y)(f(x)-f(y))\]  
for every $x\in \bar{A}$ and $f$ from its domain $\widetilde{F}_{\bar{A}}$, defined as in Section \ref{section:Preliminaries}. Obviously we have $\operatorname{ker}(\widetilde{L}_{\bar{A}})=\operatorname{span}(1)$. For $x\in\bar{A}$ and $f\in \widetilde{F}_{\bar{A}}$ we define 
\[\Delta_A f(x)=\begin{cases}
      \widetilde{L}_{\bar{A}}f(x),&x\in A\\
      0,&x\in\partial_v A.
\end{cases}\]
 
A usual way to introduce a notion of normal derivative at the vertex boundary $\partial_v A$ is to define it in terms of a Gauss-Green  formula, see for instance \cite{JP13, K17}.

 \begin{definition}
  Let $u\in D(Q_{\bar{A}})$. We say that $\varphi\in L^2(\partial_v A,m)$ is a \emph{normal derivative of $u$} if for every $v\in D(Q_{\bar{A}})$ one has 
\begin{equation}\label{E:defnormal} 
 Q_{\bar{A}}(u,v)-\langle \Delta_A u,v\rangle_{\ell^2(A,m)}=\langle \varphi, v\rangle_{\ell^2(\partial_v A,m)}.
\end{equation} 
 \end{definition}

Since $D(Q)|_{\partial_v X}$ is dense in $\ell^2(\partial_v A,m)$ one easily deduces that if $\varphi\in L^2(\partial_v A,m)$ is a normal derivative of a function $u\in D(Q_{\bar{A}})$, then it is the only normal derivative of $u$. In this case we denote this unique normal derivative by $\frac{\partial}{\partial n}u$. The normal derivative can be evaluated pointwise. 
\begin{lemma}
 Suppose $u\in D(Q_{\bar{A}})$ admits a normal derivative. Then for every $x\in\partial_v A$ we have 
 $$ \frac{\partial}{\partial n} u(x)=\frac{1}{m(x)}\sum_{y\in \bar{A}} b_{\bar{A}}(x,y)(u(x)-u(y))=\frac{1}{m(x)}\sum_{y\in A} b_{\bar{A}}(x,y)(u(x)-u(y)).$$
\end{lemma}

\begin{proof} An application of (\ref{E:defnormal}) with $v=1_x$ yields $Q_{\bar{A}}(u,1_x)= \frac{\partial}{\partial n} u(x)m(x)$, and by the definition of $Q_{\bar{A}}$ this is seen to be
\[\frac12\sum_{y,z\in\bar{A}}b_{\bar{A}}(y,z)(u(y)-u(z))(1_{x}(y)-1_x(z))=\sum_{y\in\bar{A}}b_{\bar{A}}(x,y)(u(x)-u(y)),\]
what shows the first equality. For the second use $b_{\bar{A}}(x,y)=0$,  $y\in \partial_v A$.
\end{proof}

\begin{definition} Let $\varphi\in L^2(\partial_v A,m)$. We say that a function $u\in D(Q_{\bar{A}})$ \emph{solves the Neumann problem for $\varphi$} if \[\Delta_A u(x)=0\quad \text{ and }\quad \frac{\partial}{\partial n}u(y)=\varphi(y)\] 
hold for every $x\in A$ and $y\in \partial_v A$. 
\end{definition}

Similarly as in classical analysis the following holds.

\begin{lemma}\label{lemma:uniqueness_vertex}
 Let $\varphi\in \ell^2(\partial_v A,m)$. Then two solutions of the Neumann problem for $\varphi$ differ only by a constant.  
\end{lemma}
\begin{proof}
The difference $u-v$ of two solutions $u,v$ solves the Neumann problem for zero boundary data, what implies $u-v\in\operatorname{ker}(\widetilde{L}_{\bar{A}})$, hence $u-v$ is constant.
\end{proof}

Also the following is as in the classical case.
\begin{proposition}\label{proposition:Existence_solution_finite_graph}
 Let $\varphi\in L^2(\partial_v A, m)$. If the Neumann problem has a solution, then the equality $\sum_{x\in\partial_v A} \varphi(x) m(x)=0$ holds.
\end{proposition}
\begin{proof}
Suppose $u$ solves the Neumann problem. Then (\ref{E:defnormal}) and (\ref{E:saop}) imply
$$0=Q_{\bar{A}}(u,\mathbf{1}_{\bar{A}})-\langle \Delta_A u,\mathbf{1}_{\bar{A}}\rangle_{\ell^2(A,m)}=\langle \varphi, \mathbf{1}_{\bar{A}}\rangle_{\ell^2(\partial_v A,m)} =\sum_{x\in\partial_v A} \varphi(x) m(x).\qedhere$$

\end{proof}

We give a sufficient condition for a function to be a solution of the Neumann problem for bounded $\varphi$. To do so, consider the space
 $$\mathcal{N}=\left\{u\in D(L_{\bar{A}})\cap \ell^\infty(\bar{A})\colon  L_{\bar{A}} u\in\ell^\infty(X)\right\},$$ 
it can be seen as a space of test functions.  
 \begin{lemma}\label{lemma:sufficientvertex}
  The space $\mathcal{N}$ is dense in $D(Q_{\bar{A}})$ with respect to $\|\cdot\|_{Q_{\bar{A}}}$. 
 \end{lemma}
 
 \begin{proof}
Given $v\in D(L_{\bar{A}})$ consider $f=(L_{\bar{A}}+1)v\in \ell^2({\bar{A}},m)$ and 
 choose a sequence $(f_n)_n\subset\ell^\infty({\bar{A}})$ such that $\lim_n\|f_n-f\|_{\ell^2({\bar{A}},m)}=0$. 
 Then we have $(L_{\bar{A}}+1)^{-1}f_n\in\mathcal{N}$, and by continuity ot the resolvent operators, $\lim_n\|(L_{\bar{A}}+1)^{-1}(f_n-f)\|_{\ell^2({\bar{A}},m)}=0$. Therefore
 \begin{align*}\|(L_{\bar{A}}+1)^{-1}(f_n-f)\|_{Q_{\bar{A}}}^2=&\|(L_{\bar{A}}+1)^{-1}(f_n-f)\|_{\ell^2({\bar{A}},m)}\\&+\langle f-f_n,(L_{\bar{A}}+1)^{-1}(f_n-f)\rangle_{\ell^2({\bar{A}},m)}
 \end{align*} 
 tends to zero as $n\to\infty$, showing that $\mathcal{N}$ is $\|\cdot\|_{Q_{\bar{A}}}$-dense in $D(L_{\bar{A}})$ and therefore also in $D(Q_{\bar{A}})$.
 \end{proof}

 \begin{lemma}\label{lemma:char_solution_vertex_boundary}
 Let $\varphi\in\ell^\infty(\partial_v A)$ such that  $\sum_{x\in\partial_v A} \varphi(x) m(x)=0$. 
 Then $u\in \ell^\infty({\bar{A}})$ is a solution of the Neumann problem for $\varphi$ if for every $v\in\mathcal{N}$ one has 
 $$\langle u,L_{\bar{A}}v\rangle_{\ell^2({\bar{A}},m)}=\sum_{x\in \partial_v A} \varphi(x) v(x)m(x).$$ 
 \end{lemma} 
 
 \begin{proof}
 We first show $u\in D(Q_{\bar{A}})$. For every $\alpha>0$ we observe $$L_{\bar{A}}\alpha(L_{\bar{A}}+\alpha)^{-1} u=\alpha u-\alpha^2(L_{\bar{A}}+\alpha)^{-1} u\in \ell^\infty({\bar{A}}),$$ 
 hence, $\alpha(L_{\bar{A}}+\alpha)^{-1}u\in\mathcal{N}$. The hypothesis on $u$ then implies that 
 \begin{align*} \int_{\partial_v A} \varphi \alpha(L_{\bar{A}}+\alpha)^{-1}u~dm&
 =\langle u,L_{\bar{A}}\alpha(L_{\bar{A}}+\alpha)^{-1}u\rangle_{\ell^2({\bar{A}},m)}\\
 &=\alpha\langle u,u-\alpha(L_{\bar{A}}+\alpha)^{-1}u\rangle_{\ell^2({\bar{A}},m)}\end{align*}  for
 every $\alpha>0$. Because $\alpha(L_{\bar{A}}+\alpha)^{-1}$ is $\|\cdot\|_\infty$-contractive, c.f. \cite[Theorem 1.4.1]{FOT} or \cite[Theorem 1.4.1]{D89}, the left hand side is bounded by $m(\partial_v A)\|\varphi\|_\infty\|u\|_\infty$, hence 
 $\lim\limits_{\alpha\to\infty}\alpha\langle u,u-\alpha(L_{\bar{A}}+\alpha)^{-1}u\rangle_{\ell^2({\bar{A}},m)}<+\infty$. Consequently $u\in D(Q_{\bar{A}})$ as claimed, see \cite[Lemma 1.3.4]{FOT}. By (\ref{E:saop}),
\begin{equation}\label{E:bdterms}
Q_{\bar{A}}(u,v)=\int_{\partial_v A} \varphi v~dm
\end{equation}
for every $v\in \mathcal{N}$. For general $v\in D(Q_{\bar{A}})$ let $(v_n)_n\subset \mathcal{N}$ be such that $\lim_n\|v_n-v\|_{Q_{\bar{A}}}=0$. Then 
$\lim_n Q_{\bar{A}}(u,v_n)= Q_{\bar{A}}(u,v)$ and since $\varphi\in \ell^\infty(\bar{A},m)$ also $\lim_n \int_{\partial_v A}|\varphi v_n -\varphi v |~dm= 0$,
what shows that (\ref{E:bdterms}) holds for all $v\in D(Q_{\bar{A}})$. Finally, note that for every $x\in A$ we have 
\[\Delta_A u(x)=\frac{1}{m(x)}Q_{\bar{A}}(u,1_x)=\frac{1}{m(x)}\int_{\partial_v A} \varphi 1_x~dm=0.\qedhere\]
 \end{proof}

Our main interest is to provide representation formulas for solutions to the Neumann problem in terms of heat kernels, green functions and local times, c.f. \cite[Theorem 3.10.]{brosamler}, \cite[Theorem 5.3.]{bass_hsu}, the existence of solutions will follow as a by-product. 

Let $Y=(Y_t)_{t\geq 0}$ be the Markov process (uniquely) associated
with $Q_{\bar{A}}$, cf. Appendix \ref{appendix:Dirichlet}. Then $e^{-tL_{\overline{{A}}}} f(x)=\IE_x f(Y_t)$ for every $f\in \ell^2(\bar{A},m)$ and all $x\in \bar{A}$. The process $(Y_t)_{t\geq 0}$ is a continuous time Markov process with state space $\bar{A}$ and waiting time distribution at point $x\in\bar{A}$ being exponential with parameter $\frac{1}{m(x)}\sum_{y\in\bar{A}}b_{\overline {A}}(x,y)$ and single step transition probabilities $\frac{b_{\bar{A}}(x,y)}{\sum_{z\in\bar{A}} b_{\bar{A} (x,y)}}$. 

\begin{definition}
 We define $L_t:=\int_0^t \mathbf{1}_{\{Y_s\in \partial_v A\}} ds$, $t\geq 0$ and call the process $(L_t)_{t\geq 0}$ the \emph{local time} of $Y$ on the vertex boundary $\partial_v A$. 
\end{definition}
Given any finite time interval, the probability that the process $(L_t)_{t\geq 0}$ jumps only finitely many times in this interval is one.
Consequently the function $(L_t)_{t\geq 0}$ is almost everywhere differentiable with derivative $(L_s)':=\frac{d}{dt}|_{t=s}L_t=\mathbf{1}_{\{Y_s\in \partial_v A\}}$.
The process $(L_t)_{t\geq 0}$ is a positive continuous additive functional of $Y$, it is in Revuz correspondence with the measure $m|_{\partial_v A}$, see \cite[Section 5.1, in particular Theorem 5.1.4]{FOT}. This fact can be expressed using the heat kernel $\left\lbrace p_t(\cdot,\cdot)\right\rbrace_{t>0}$ for $e^{-tL_{\bar{A}}}$, for its definition see formula (\ref{E:hk}) in Appendix \ref{appendix:Dirichlet}.

\begin{lemma}\label{lemma:Revuz_simple}
  Let $\varphi\in \ell^\infty(\partial_v A)$. For every $t\geq 0$ the function 
  \begin{equation}\label{E:ut}
  u_t(x):=\IE_x\int_0^t \varphi(Y_s)~dL_s
  \end{equation} 
  is bounded, and $u_t(x)= \sum_{y\in\partial_v A} \int_0^t \varphi(y) p_s(x,y)m(y)~ds$ holds for every $x\in\overline{A}$.
\end{lemma}
\begin{proof}We have 

\begin{multline}
                       u_t(x)=\IE_x\int_0^t \varphi(Y_s)(L_s)'~ds
                       =\IE_x\int_0^t \varphi(Y_s)\mathbf{1}_{\{Y_s\in \partial_v A\}}~ds
                       =\sum_{y\in\partial_v A} \IE_x\int_0^t \varphi(Y_s)\mathbf{1}_{\{Y_s=y\}}~ds\\
                       =\sum_{y\in\partial_v A} \IE_x\int_0^t \varphi(y)\mathbf{1}_{\{Y_s=y\}}~ds
                       =\sum_{y\in\partial_v A} \varphi(y)\IE_x\int_0^t \mathbf{1}_{\{Y_s=y\}}~ds,\notag
\end{multline}
where the first equality follows since $L_\cdot$ is pathwise almost everywhere differentiable. Fubini-Tonelli yields $\IE_x\int_0^t \mathbf{1}_{\{Y_s=y\}}~ds=\int_0^t \IE_x \mathbf{1}_{\{Y_s=y\}}~ds$, and using $p_t(x,y)=\frac{1}{m(y)}\IE_x \mathbf{1}_{\{Y_t=y\}}$ we obtain
\begin{multline}
 u_t(x)=\sum_{y\in\partial_v A} \varphi(y)\int_0^t e^{-s L_{\bar{A}}}1_y(x)~ds
 =\sum_{y\in\partial_v A} \int_0^t\varphi(y) e^{-s L_{\bar{A}}}1_y(x)~ds\notag\\
 =\sum_{y\in\partial_v A} \int_0^t \varphi(y) p_s(x,y)m(y)~ds,\notag
 \end{multline}
which admits the bound $|u_t(x)|\leq t\|\varphi\|_\infty$.  
\end{proof}

To establish representation formulas we adapt the method of \cite {bass_hsu} to graphs, and to do so, we make one more standing assumption. It is always satisfied if $\bar{A}$ is a finite set. See Appendix \ref{appendix:Dirichlet} for further information.

\begin{assumption}\label{A:ultravertex}
The semigroup $(e^{-tL_{\overline{{A}}}})_{t>0}$ is ultracontractive.
\end{assumption}

The \emph{Green kernel} $g:\bar{A}\times\bar{A}$ is defined by
\begin{equation}\label{E:defgreen}
g(x,y)=\int_0^\infty \left(p_t(x,y)-\frac{1}{m(\bar{A})}\right)dt,\quad x,y\in\bar{A}.
\end{equation}  
The existence of this integral is a consequence of Assumption \ref{A:ultravertex}, see formula (\ref{E:mixing}) in Appendix \ref{appendix:Dirichlet}.

We can now state the main result of this section, recall that Assumptions \ref{A:basicdiscrete} and \ref{A:ultravertex} must be satisfied.

\begin{theorem}\label{theoremvertexboundary}
 Let $\varphi\in \ell^\infty(\partial_v A)$ be such that $\sum_{x\in \partial_v A} \varphi(x) m(x)=0$. 
 \begin{itemize} 
  \item [(i)] There is a unique function $u:\bar{A}\to\IR$ that solves the Neumann problem for $\varphi$ and satisfies $\sum_{x\in \bar{A}} u(x) m(x)=0$.
  \item[(ii)] For this function $u$  and for every $x\in\bar{A}$ we have 
  $$u(x)=\lim\limits_{t\to\infty}\IE_x\int_0^t \varphi(X_s)~dL_s=\int_0^\infty \sum_{y\in\partial_v A} \varphi(y) p_s(x,y)m(y)~ds
 = \sum_{y\in\partial_v A} \varphi(y) g(x,y)m(y). $$
 \end{itemize}
 \end{theorem}

The first equality in (ii) is a graph version of \cite[Theorem 5.3]{bass_hsu} and \cite[formula (1.1)]{brosamler}. Our proof follows these references.

\begin{proof}
We first show that $u(x):=\lim_{t\to\infty} u_t(x)$ exists. By the previous lemma and since $\varphi$ is centered, 
 \begin{align*}
  u_t(x)=\sum_{y\in\partial_v A} \int_0^t \varphi(y) p_s(x,y)m(y)~ds=\sum_{y\in\partial_v A} \int_0^t \varphi(y) \left(p_s(x,y)-\frac{1}{m(\bar{A})}\right)m(y)~ds,
 \end{align*}
and given $0<s<t$ we can use (\ref{E:mixing}) to see that
\begin{align*}
 |u_t(x)-u_s(x)|&\leq \sum_{y\in\partial_v A} \int_s^t |\varphi(y)|\left|p_r(x,y)-\frac{1}{m(\bar{A})}\right|m(y)~dr\\
 &\leq c_1 \sum_{y\in\partial_v A}|\varphi(y)|m(y)\int_s^te^{-c_2 r}~dr,\notag
\end{align*}
what becomes arbitrarily small if only $s,t$ are chosen large enough. Hence, the limit function $u$ exists, is bounded, and satisfies $$u(x)=\int_0^\infty \sum_{y\in\partial_v A} \varphi(y) p_s(x,y)m(y)~ds$$ for every $x\in\bar{A}$.
By (\ref{E:defgreen}) and since $\varphi$ is centered, we have $u(x)=\sum_{y\in\partial_v A} \varphi(y) g(x,y)m(y)$ for every $x\in\bar{A}$, and integrating,
\[\sum_{x\in \bar{A}} u(x) m(x)=\int_0^\infty\sum_{y\in \partial_v A}\varphi(y)\sum_{x\in \bar{A}} p_s(x,y)m(x)m(y)~ds.
\]
By (\ref{E:stochcompl}) and since $\varphi$ is centered, we see that $\sum_{x\in \bar{A}} u(x) m(x)=0$. 
We show that $u$ solves the Neumann problem by an application of Lemma~\ref{lemma:char_solution_vertex_boundary}. Let $v\in \mathcal{N}$ be arbitrary. 
We have to show $\langle u,L_{\bar{A}}v\rangle_{\ell^2({\bar{A}},m)}=\sum_{x\in \partial_v A} \varphi(x) v(x)m(x)$.
The uniform convergence of $u_t$ to $u$ yields
$$\langle u,L_{\bar{A}}v\rangle_{\ell^2({\bar{A}},m)}=\lim_{t\to\infty}\langle u_t,L_{\bar{A}}v\rangle_{\ell^2({\bar{A}},m)}.$$
On the other hand
\begin{align*}
 \langle u_t,L_{\bar{A}}v\rangle_{\ell^2({\bar{A}},m)}&=\sum_{y\in\partial_v A} \int_0^t \varphi(y)\sum_{x\in\bar{A}} p_s(x,y)L_{\bar{A}}v(x)m(x)m(y)~ds\\
 &=\sum_{y\in\partial_v A} \int_0^t \varphi(y)\langle p_s(\cdot,y),L_{\bar{A}}v\rangle_{\ell^2(\bar{A},m)}m(y)~ds.
\end{align*}
Using of the self-adjointness of $L_{\bar{A}}$ and the heat equation (\ref{E:heateqforp}),
$$\langle p_s(\cdot,y),L_{\bar{A}}v\rangle_{\ell^2(\bar{A},m)}=-\langle \frac{d}{dt}p_t(\cdot,y)|_{t=s}  ,v\rangle_{\ell^2(\bar{A},m)},$$
and therefore
\begin{align*}
 \langle u_t,L_{\bar{A}}v\rangle_{\ell^2({\bar{A}},m)}
 &=-\sum_{y\in\partial_v A} \int_0^t \sum_{x\in\bar{A}} \varphi(y) (\frac{d}{dt} p_t(\cdot,y)|_{t=s} )(x) v(x)m(x)m(y)~ds\\
 &=\sum_{y\in\partial_v A} \sum_{x\in\bar{A}} \varphi(y)  (p_0(x,y)-p_t(x,y))v(x)m(x)m(y).
\end{align*}
%
By Lebesgue's theorem and (\ref{E:mixing}), 
$$
\lim_{t\to\infty}\sum_{y\in\partial_v A}\sum_{x\in\bar{A}} \varphi(y)p_t(x,y)v(x)m(x)m(y)
=\sum_{y\in\partial_v A} \varphi(y)m(y)\frac{1}{m(\bar{A})}\sum_{x\in\bar{A}} v(x)m(x)=0,
$$
so that
\[\langle u,L_{\bar{A}}v\rangle_{\ell^2({\bar{A}},m)}= \sum_{y\in\partial_v A} \sum_{x\in\bar{A}} \varphi(y)p_0(x,y)v(x)m(x)m(y)=\sum_{y\in\partial_v A} \varphi(y)v(y)m(y).\]
The result now follows from Lemma~\ref{lemma:char_solution_vertex_boundary}.
%
%
%
\end{proof}

\section{Neumann problems for canonically compactifiable graphs}\label{section:brosamler_royden}

We consider canonically compactifiable graphs, a class of graphs introduced and studied in \cite{canon}, and investigate Neumann problems with respect to the abstract Royden boundary. Some
information on this boundary can be found in Appendix~\ref{appendix:royden}, for more details we refer to  \cite{Soardi}.

\begin{definition}
 A connected graph $b$ over a countably infinite set $X$ is called \emph{canonically compactifiable} if the inclusion $\widetilde{D}\subseteq \ell^\infty(X)$ holds. 
\end{definition}

From now on we assume the following.

\begin{assumption}\label{A:canon}
The graph $b$ over $X$ is canonically compactifiable and $m$ is a finite measure on $X$ such that $(X,m)$ is a finite discrete measure space.
\end{assumption}

Under this assumption we have $D(Q)=\widetilde{D}$, and according to \cite[Theorem 5.1.]{canon}  the semigroup $(e^{-tL})_{t>0}$ is ultracontractive with rate function $\gamma$ as in (\ref{E:ratefct}), Appendix \ref{appendix:Dirichlet}.

Let $R$ be the Royden compactification of $X$, see Appendix~\ref{appendix:royden}, and denote the continuous extension of every $f\in D(Q)$ to $R$ by $\widehat{f}$. 
On the Borel-$\sigma$-field $\mathcal{B}(R)$ we define a measure 
$\widehat{m}$ via 
\begin{equation}\label{E:transferm}
\widehat{m}(A):=m(X\cap A).
\end{equation}
It is easy to see that $\widehat{m}$ is a Radon-measure with full support. 

The operation $D(Q)\to C(R),$ $f\mapsto \widehat{f}$ can be used to introduce a natural Dirichlet form on $L^2(R,\widehat{m})$, \cite{Soardi, KLSW17}, see also \cite{HKT, H} for related results. To do so denote by $[\widehat{f}]$ the equivalence class of functions that coincide $\widehat{m}$-a.e. with $\widehat{f}$. We can define a Dirichlet form $\widehat{Q}$ on $L^2(R,\widehat{m})$ by $$D(\widehat{Q})=\{[\widehat{f}]\colon  f\in D(Q)\},\quad\widehat{Q}([\widehat{f}])=Q(f).$$
By construction, $\widehat{Q}$ is regular, cf. Appendix~\ref{appendix:Dirichlet} for the definition. Notationally we will not distinguish between $\widehat{f}$ and $[\widehat{f}]$. We write $\widehat{L}$ to denote the generator of $\widehat{Q}$.

Canonically compactifiable graphs satisfy a strong Sobolev embedding.
\begin{lemma}\label{lemma:embedding_Neumann_supnorm}
There is a constant $c>0$ such that $\|u\|_\infty\leq c\|u\|_Q$ holds for every $u\in D(Q)$. In particular, $\|\widehat{u}\|_\infty\leq c\|\widehat{u}\|_{\widehat{Q}}$ holds for every $\widehat{u}\in D(\widehat{Q})$.
\end{lemma}
\begin{proof}
 The first inequality is proven in \cite[Lemma 4.2., Lemma 5.3.]{canon}. The second inequality follows 
 by $\|u\|_\infty=\|\widehat{u}\|_\infty$ (as $X$ is dense in $R$) and $Q(u)=\widehat{Q}(\widehat{u})$ (by definition).
\end{proof}

Let $\left\lbrace p_t(x,y)\right\rbrace_{t>0}$ be the heat kernel for the semigroup $(e^{-tL})_{t>0}$, cf. Appendix \ref{appendix:Dirichlet}. For every $t>0$ the function $p_t(\cdot,\cdot)$ has a continuous extension to $R\times R$. To see this, note that (by symmetry) for fixed $x\in X$ both 
$ p_t(x,\cdot)$ and $ p_t(\cdot,x)$ are elements of $D(Q)$ and therefore have unique continuous extensions $\widehat{p_t(x,\cdot)}$ and $\widehat{p_t(\cdot,x)}$ to $R$. By (\ref{E:CK}) and dominated convergence there is a unique continuous extension of $(x,y)\mapsto p_t(x,y)$ to a symmetric function on $R\times R$, we denote it by 
\[(x,y)\mapsto \widehat{p}_t(x,y).\] 
Lemma \ref{lemma:embedding_Neumann_supnorm} can be used to see that the $\widehat{p}_t(\cdot,\cdot)$ determine a Feller transition function, cf. Appendix~\ref{appendix:revuz}. For any $x\in R$ consider the Borel measure 
\begin{equation}\label{E:defPhat}
\widehat{P}_t(x,A)=\int_{A} \widehat{p}_t(x,y)~\widehat{m}(dy), \quad A\in\mathcal{B}(R).
\end{equation}
The next lemma implies that for $x\in X$ we have $e^{-tL}\mathbf{1}_A(x)=\widehat{P}_t(x,X\cap A)$,  $\quad A\in\mathcal{B}(R)$, which is similar to (\ref{E:transferm}). However, (\ref{E:defPhat}) makes sense for any $x\in R$ and in this sense (\ref{E:defPhat}) 
is a slight abuse of notation.

\begin{lemma}\label{lemma:heat_kernel_feller}
We observe the following.
\begin{itemize}
 \item [(i)] The kernels $\widehat{P}_t$ form a Feller transition function $(\widehat{P}_t)_{t>0}$ on $X$. It satisfies $\widehat{P}_t\mathbf{1}=\mathbf{1}$ for all $t>0$  and it has the strong Feller property. 
 \item[(ii)] The Feller transition function $(\widehat{P}_t)_{t>0}$ is in correspondence with $\widehat{Q}$, i.e. for every $u\in L^2(R,\widehat{m})$ and for every representative $v$ of $u$ and for every $t>0$ the function $\widehat{P}_t v$ is a representative of 
$e^{-t\widehat{L}} u$. Moreover, for any such $u$ and $v$ we have 
\begin{equation}\label{E:extends}
\widehat{P}_t v(x)=\widehat{e^{-tL}u|_X}(x) 
\end{equation}
for every $x\in R$. The family $\left\lbrace \widehat{p}_t(\cdot,\cdot)\right\rbrace_{t>0}$ is a heat kernel for $(e^{-t\widehat{L}})_{t>0}$. The semigroup $(e^{-t\widehat{L}})_{t>0}$ is ultracontractive with rate function $\gamma$. The operators $e^{-t\widehat{L}}$ are compact and $\widehat{L}$ has pure point spectrum.
\item[(iii)] There is a Hunt process $Y=(Y_t)_{t\geq 0}$ on $R$ in correspondence with $(\widehat{P}_t)_{t>0}$, i.e. such that $\widehat{P}_t(x,A)=\mathbb{P}_x(Y_t\in A)$ for all $x\in R$ and $A\in\mathcal{B}(R)$.
\end{itemize}
\end{lemma}

The restriction to $X$ in (ii) is well-defined since $m(x)>0$ for all $x\in X$.
\begin{proof}
First observe, using $\widehat{m}(\partial_R X)=0$, that for every $x\in R$ and $A\in \mathcal{B}(R)$ one has 
\begin{equation}\label{E:justasum}
\widehat{P}_t(x,A)=\sum_{y\in X\cap A} \widehat{p}_t(x,y)m(y).
\end{equation}
Using the finiteness of $m$, the boundedness of $\widehat{p}_t(\cdot,\cdot)$, and the continuity of
$\widehat{p}_t(\cdot,y)$ for every fixed $t>0$ and $y\in R$, dominated convergence yields the continuity of $\widehat{P}_tv$ for every bounded measurable function $v$ on $R$. In particular, for any $A\in\mathcal{B}(R)$ the map $\widehat{P}(\cdot,A)$ is measurable. Moreover, it is easy to see that for every $x\in X$ the map $\widehat{P}_t(x,\cdot)$ is a measure. For $x\in X$ one has $\widehat{P}_t(x,R)=\sum_{y\in X} p_t(x,y)m(x)=1$ and, using continuity, $\widehat{P}_t(x,R)=1$ for every $x\in R$. To verify the semigroup property, suppose that $s,t>0$ and that $v$ is a bounded and measurable function on $R$. For $x\in X$ we have $\widehat{P}_t\widehat{P}_s v(x)=\widehat{P}_{t+s} v(x)$ by (\ref{E:CK}), for $x\in \partial_R X$ it follows from the continuity of $\widehat{P}_t v$ on $R$. We verify that for $u\in C(R)$ we have  $\lim_{t\to 0}\|\widehat{P}_tu -u\|_\infty=0$. Since $D(\widehat{Q})$ is $\|\cdot\|_\infty$-dense in $C(R)$ (see Appendix~\ref{appendix:royden}) it suffices to show the convergence for every $\widehat{u}\in D(\widehat{Q})$. By \cite[Lemma 1.3.3]{FOT} we have $\lim_{t\to 0}\|e^{-tL} u-u\|_{Q}= 0$ and by Lemma~\ref{lemma:embedding_Neumann_supnorm} therefore $\lim_{t\to 0}\|e^{-tL} u-u\|_{\infty}= 0$. Because $\widehat{P}_t \widehat{u}$ is the continuous extension of $e^{-tL} u$ we obtain 
\[\lim_{t\to 0}\sup_{x\in R} |\widehat{P}_t \widehat{u}(x)-\widehat{u}(x)|= 0.\]
This shows (i). Let $u\in L^2(R,\widehat{m})$ be arbitrary and $v$ a representative of $u$. Then 
\begin{align*}
\int_{R}(\widehat{P}_t v)^2~d\widehat{m}=\sum_{x\in X} \left(\sum_{y\in X}u(y)p_t(x,y)m(y)\right)^2 m(x)=\sum_{x\in X} (e^{-tL} (u|_X))^2 m(y)<+\infty,
\end{align*}
so that $\widehat{P}_t$ gives rise to an operator on $L^2(R,\widehat{m})$. A similar calculation shows (\ref{E:extends}) for every $x\in X$ and the continuity of $\widehat{P}_t v$ then implies (\ref{E:extends}) for all $x\in R$. As the boundary has measure zero, $\widehat{P}_tv$ is a representative of 
$e^{-t\widehat{L}}u$. That the functions $\widehat{p}_t(\cdot,\cdot)$ form a heat kernel is obvious. The ultracontractivity of the semigroup $(e^{-t\widehat{L}})_{t>0}$ is easy to see, note that for $u\in L^2(R,\widehat{m})$ we have 
\[\big\|\widehat{P}_tu\big\|_{L^\infty(R,\widehat{m})}=\left\|e^{-tL}u|_X\right\|_{\ell^\infty(X)}\leq \gamma(t)\left\|u|_X\right\|_{\ltwo}=\gamma(t)\left\|u\right\|_{L^2(R,\hat{m})}.\]
For the remaining statements in (ii) see \cite[Theorems 2.1.4 and 2.1.5]{D89}. Item (iii) follows from \cite[Chapter I, (9.4) Theorem]{BG68} or \cite[Theorem A.2.2]{FOT}, as $R$ is separable and compactly metrizable, see Lemma \ref{lemma:royden_metrizable}.
\end{proof}

To define normal derivatives for the Royden boundary we follow \cite{Maeda64}. Let $\mu$ be an arbitrary Radon measure on $\mathcal{B}(\partial_R X)$. Note that as $\partial_R X$ is compact, the measure $\mu$ is automatically finite.

\begin{definition}
Let $\varphi\in L^2(\partial_R X, \mu)$ and $u\in D(Q)$ with $\widetilde{L} u\in \ell^2(X,m)$. We say that $\varphi$ is the \emph{normal derivative} of $u$ if for every $v\in D(Q)$ one has 
\begin{equation}\label{E:defnormalRoyden}
Q(u,v)-\langle \widetilde{L} u, v \rangle_{\ell^2(X,m)}=\int_{\partial_R X} \varphi \widehat{v}~d\mu.
\end{equation}
\end{definition}
If $\varphi$ is a normal derivative for $u\in D(Q)$ then it is easily seen to be unique, and we denote it by $\frac{\partial}{\partial n} u=\varphi$. By (\ref{E:saop}) every $u\in D(L)$ has a normal derivative given by $0$.

We formulate the Neumann problem.
\begin{definition}
Let $\varphi:\partial_R X\to\IR$ be measurable and bounded. We say that $u\in D(Q)$ is a \emph{solution of the Neumann problem} for $\varphi$ if 
\[\widetilde{L} u\equiv 0\quad \text{ and }\quad \frac{\partial}{\partial n}u=\varphi\] 
hold, where the latter equality means that $\varphi$ is a representative of $\frac{\partial}{\partial n}u$.
\end{definition}

The next lemma discusses uniqueness, its proof is the same as the proof for Lemma \ref{lemma:uniqueness_vertex}.
\begin{lemma}\label{lemma:uniqueness}
Let $\varphi:\partial_R X\to\IR$ be measurable. Then, two solutions $u_1,u_2$ of the Neumann problem differ only by a constant. 
\end{lemma}

We give a necessary criterion for the existence of a solution.
\begin{proposition}\label{proposition:existence_royden_centered}
Let $\varphi:\partial_R X\to\IR$ be measurable and bounded and let $u\in D(Q)$ be a solution of the Neumann problem for $\varphi$. Then, $\varphi$ satisfies $$\int_{\partial_R X} \varphi~d\mu=0.$$  
\end{proposition}
\begin{proof}
Let $v\equiv 1$. Then, $v$ satisfies $Q(u,v)= 0$ and $\frac{\partial}{\partial n}v=0$. By definition, 
\[\int_{\partial_R X} \varphi ~d\mu=\sum_{x\in X} \widetilde{L}u(x)v(x)m(x)=0.\qedhere\]
\end{proof}

Similarly as before we give a sufficient criterion for a function to be a solution of the Neumann problem for bounded $\varphi$. Set
 $$\mathcal{N}=\left\{u\in D(Q)\colon \widetilde{L}u\in\ell^\infty(X),~\frac{\partial}{\partial n}u=0\right\}.$$ 
 By (\ref{E:defnormalRoyden}) we have $\mathcal{N}\subseteq D(L)$. The next lemma can be proved like Lemma \ref{lemma:sufficientvertex}. 
 \begin{lemma}
  The space $\mathcal{N}$ is dense in $D(Q)$ with respect to $\|\cdot\|_Q$. 
 \end{lemma}

The Radon measure $\mu$ on $\mathcal{B}(\partial_R X)$ can also be seen as a Radon measure on $\mathcal{B}(R)$. By  Lemma~\ref{lemma:embedding_Neumann_supnorm} it is of finite energy integral (see \cite[Section 2.2]{FOT}), i.e. there is some $c>0$ such that for all $\widehat{u}\in D(\widehat{Q})$ we have
$$\int_R |\widehat{u}|~d\mu\leq c\|\widehat{u}\|_{\widehat{Q}}.$$ 

 \begin{lemma}\label{lemma:char_solution}
 Let $\varphi:\partial_R X\to\IR$ be measurable and bounded such that  $\int_{\partial_R X} \varphi~d\mu=0$. Then, a function $u\in \ell^\infty(X)$ is a solution of the Neumann problem for $\varphi$ if for every $v\in\mathcal{N}$ one has $$\langle u,Lv\rangle_{\ell^2(X,m)}=\int_{\partial_R X} \varphi \widehat{v}~d\mu.$$ 
 \end{lemma} 
 
 \begin{proof}
The same arguments as in the proof of Lemma \ref{lemma:char_solution_vertex_boundary} show that $\alpha(L+\alpha)^{-1}u\in\mathcal{N}$, and therefore $u\in D(Q)$ and 
\begin{equation}\label{E:tobeshown}
Q(u,v)=\int_{\partial_R X} \varphi \widehat{v}~d\mu
\end{equation}
for every $v\in \mathcal{N}$. For general $v\in D(Q)$ let $(v_n)_n\subset\mathcal{N}$ be such that $\lim_n\|v_n-v\|_Q= 0$. Then $\lim_n Q(u,v_n)= Q(u,v)$ and, since $\widehat{\mu}$ is of finite energy integral and $\varphi$ is bounded, 
 $$\lim_n\int_{\partial_R X}|\varphi\widehat{v_n}-\varphi\widehat{v}|~d\mu=0,$$
and we see that (\ref{E:tobeshown}) holds for every $v\in D(Q)$. 
 
To see that $\widetilde{L} u\equiv 0$, note that $Q(u,v)=0$ for every $v\in C_c(X)$ and that for $v=\frac{1}{m(x)}1_x$ we have $Q(u,v)=\widetilde{L}u(x)$.
 \end{proof}

The measures of finite energy integral are in one-to-one correspondence with the \emph{positive continuous additive functionals} of the Hunt process $Y=(Y_t)_{t\geq 0}$ via the Revuz correspondence, see \cite[Chapter 5.1, Theorem 5.1.3]{FOT}. In the present setup this tells that there is a positive continuous additive functional $(L_t)_{t\geq 0}$ of $Y$, unique in the sense of \cite[Chapter 5.1]{FOT}, such that the equality 
\begin{equation}\label{E:Revuz}
\int_X h\cdot\IE_x\int_0^t f(Y_s)dL_s~dm=\int_0^t \int_X f\cdot \widehat{P}_sh~d\mu~ds
\end{equation} 
holds for every $t>0$ and every bounded measurable $f,h:R\to\IR$. 
\begin{definition}
 We call the process $(L_t)_{t\geq 0}$ the \emph{local time} of $Y$ on the Royden boundary $\partial_R X$ with respect to $\mu$. 
\end{definition}


\begin{lemma}\label{L:utRoyden}
Let $\varphi:\partial_RX\to \mathbb{R}$ be bounded and measurable. Then for every $t\geq 0$ the function 
\[u_t(x):=\mathbb{E}_x\int_0^t\varphi(Y_s)dL_s,\quad x\in X,\]
is bounded, and $u_t(x)=\int_0^t \int_{\partial_R X} \varphi(y) \widehat{p}_s(x,y)~d\mu (y) ds$ holds for every $x\in X$.
\end{lemma}
\begin{proof}
The application of (\ref{E:Revuz}) to the function $h=\frac{1_x}{m(x)}$, $x\in X$, yields that for every $x\in X$ and bounded measurable function $\varphi:\partial_R X\to\IR$ the equality 
\begin{equation}\label{E:Revuzsimple}
\IE_x\int_0^t \varphi(Y_s)~dL_s =\int_0^t \int_{\partial_R X} \varphi(y) \widehat{p}_s(x,y)~d\mu (y) ds
\end{equation} 
holds, where for every $t$ the left hand side is bounded by $\|\varphi\|_\infty \IE_x L_t$. From the proof of \cite[Lemma 5.1.9]{FOT} it can be extracted that the map $X\to\IR,$ $x\mapsto\IE_x L_t$ is an element of of $D(Q)\subseteq \ell^\infty(X)$ for every $t\geq 0$. In particular, the right hand side exists for every $t\geq 0$, as the left hand side is bounded.
\end{proof}

We define a \emph{Green kernel} $\widehat{g}:X\times R$ via 
$$\widehat{g}(x,y)=\int_0^\infty \left(\widehat{p}_t(x,y)-\frac{1}{m(X)}\right)s~dt,\quad x\in X,y\in R,$$ where the integral exists by (\ref{E:easyhke}) and (\ref{E:mixing}).

Again following  \cite{brosamler} and \cite{bass_hsu} we can now formulate a counterpart of Theorem \ref{theoremvertexboundary} for Neumann problems with respect to the Royden boundary. Recall that Assumption \ref{A:canon} must be in force.

\begin{theorem}\label{theorem:solution_royden}
Let $\varphi:\partial_R X\to\IR$ be bounded and measurable such that $\int_{\partial_R X} \varphi~d\mu=0$.
\begin{itemize}
 \item[(i)] There is a unique solution $u\in D(Q)$ of the Neumann problem that satisfies\\ $\sum_{x\in X} u(x) m(x)=0$.
 \item[(ii)] For this function $u$ and every $x\in X$ we have
 $$
 u(x)=\lim_{t\to\infty}\IE_x\int_0^t \varphi(Y_s)~dL_s=\int_0^\infty \int_{\partial_R X} \varphi(y) \widehat{p}_s(x,y)~d\mu~ds
 = \int_{\partial_R X} \varphi(y) \widehat{g}(x,y)~d\mu. 
 $$
\end{itemize}
\end{theorem}

\begin{remarks}
In \cite[Theorem 7.9.]{kasue} it was proved that the Neumann problem on the Kuramochi boundary has a solution for given $\varphi$ on the boundary if and only if $\varphi$ is centered. They used the same notion of normal derivatives and solutions of the Neumann problem, but with a fixed \emph{harmonic} measure instead of an arbitrary  Radon measure $\mu$ on the boundary. Being targeted at density questions for spaces of harmonic functions, their approach provides a general abstract existence result for solutions. It does not lead to explicit representation formulas as in part (ii) of the theorem above. 
\end{remarks}

\begin{proof}

Similarly as for the vertex boundary case we see that for any $x\in X$ the limit $u(x):=\lim_{t\to \infty} u_t(x)$ exists and equals
$\int_0^\infty \int_{\partial_R X} \varphi(y) \widehat{p}_s(x,y)~d\mu (y) ds$. The convergence is uniform and, hence, $u$ is bounded on $X$. In fact, by (\ref{E:Revuzsimple}) and since $\varphi$ is centered, we have 
$$\IE_x \int_0^t\varphi(Y_s)~dL_s=\int_0^t \int_{\partial_R X} \varphi(y) \left(\widehat{p}_s(x,y)-\frac{1}{m(X)}\right)~d\mu (y) ds,$$
and using (\ref{E:mixing}) and the boundedness of $\varphi$, 
\begin{align}
 |u_t(x)-u_s(x)|&\leq \left|\int_s^t \int_{\partial_R X} \varphi(y) \left(\widehat{p}_r(x,y)-\frac{1}{m(X)}\right)~d\mu (y) dr\right|\notag\\
&\leq \int_s^t \int_{\partial_R X} |\varphi(y)| \left|\widehat{p}_r(x,y)-\frac{1}{m(X)}\right|~d\mu (y) dr\notag\\
&\leq c_1 \int_{\partial_R X} |\varphi(y)| d\mu(y)\:\int_s^t e^{-c_2r}~dr,\notag
\end{align}
which can be made arbitrarily small. Using the definition of the Green kernel $g$ and the fact that $\varphi$ is centered, 
\[u(x)=\int_{\partial_R X} \varphi(y) \widehat{g}(x,y)~d\mu,\quad x\in X.\]
We next observe that 
\begin{align*}
\sum_{x\in X} u(x) m(x)&=\sum_{x\in X}m(x)\int_0^\infty \int_{\partial_R X} \varphi(y) \left(\widehat{p}_s(x,y)-\frac{1}{m(X)}\right)~d\mu (y) ds\\
&=\int_0^\infty \int_{\partial_R X} \varphi(y)\sum_{x\in X} \left(\widehat{p}_s(x,y)m(x)-\frac{1}{m(X)}m(x)\right)~d\mu (y) ds\\
&=\int_0^\infty \int_{\partial_R X} \varphi(y)\left(\int_R \widehat{p}_s(x,y)~d\widehat{m}(x)-1\right)~d\mu (y) ds\\
&=0.
\end{align*}
%

Using  Lemma~\ref{lemma:char_solution} we show that $u$ is a solution of the Neumann problem. To do so we verify that for 
 $v\in\mathcal{N}$ and every $t>0,x\in R$ we have 
 \begin{equation}\label{E:claim}
\widehat{v}(x)-\widehat{P}_t\widehat{v}(x)=\int_0^t \widehat{e^{-tL}(Lv)}(x)~ds.
\end{equation}
Suppose that  $x\in X$. The symmetry of $p_t(x,y)$ and the heat equation (\ref{E:heateqforp}) imply that 
\[-\sum_{y\in X} \frac{d}{ds}p_s(x,y)|_{s=t} v(y) m(y)=\sum_{y\in X} p_t(x,y) Lv(y) m(y).\]
Integrating,
\begin{align}
\widehat{v}(x)-\widehat{P}_t\widehat{v}(x)&=-\sum_{y\in X} (p_t(x,y)-p_0(x,y))v(y) m(y)\notag\\
&=\int_0^t\sum_{y\in X} p_s(x,y) Lv(y) m(y)~ds\notag\\
&=\int_0^te^{-sL}(Lv)(x)ds.\notag
\end{align}
Given $x\in \partial_R X$ let $(x_n)_n\subset X$ be such that $\lim_n x_n=x$. By continuity
\[\lim_{n\to \infty}\widehat{v}(x_n)-\widehat{P}_t \widehat{v}(x_n)= \widehat{v}(x)-\widehat{P}_t \widehat{v}(x).\] 
Since $\ell^\infty(X)\subset \ell^2(X,m)$ and $Lv\in \ell^\infty(X)$ we can use the contractivity of the operators $e^{-tL}|_{\ell^\infty(X)}$ from $\ell^\infty(X)$ to $\ell^\infty(X)$ to see that   
\[|e^{-tL}(L v)(x)|\leq \|L v\|_\infty,\]
see \cite[Theorem 1.4.1]{D89}. Therefore dominated convergence and the continuity of $\widehat{e^{-tL}(Lv)}$ for any $t>0$ show that 
\[\lim\limits_{n\to\infty}\int_0^t e^{-sL}(Lv)(x_n)~ds=\int_0^t \lim\limits_{n\to\infty} e^{-sL}(Lv)(x_n)~ds=\int_0^t  \widehat{e^{-sL}(Lv)}(x)~ds,\]
so that (\ref{E:claim}) follows.
Let $v\in \mathcal{N}$ be arbitrary. By Fubini,
\[\sum_{x\in X} u_t(x)L v(x)m(x)=\int_{\partial_R X}\varphi(y)\int_0^t \sum_{x\in X}\widehat{p}_s(x,y)Lv(x)m(x)~ds~d\mu(y),\]
and by (\ref{E:claim}),
\begin{align*}
\sum_{x\in X} u_t(x)L v(x)m(x)=\int_{\partial_R X}\varphi(\widehat{v}-\widehat{P}_t\widehat{v})d\mu.  
\end{align*}
From (\ref{E:mixing}) we obtain $\lim_{t\to\infty} \widehat{P}_t\widehat{v}=\frac{1}{m(X)}\int_X\widehat{v}~dm$ uniformly on $R$. Thus, the equality  $\lim_{t\to\infty} \int_{\partial_R X} \varphi \widehat{P}_t\widehat{v}~d\mu=0$ follows.
Hence, using the uniform convergence of $u_t$ to $u$, we deduce $$\sum_{x\in X} u(x)L v(x)m(x)=\lim_{t\to\infty}\sum_{x\in X} u_t(x)L v(x)m(x)=\int_{\partial_R X} \varphi \widehat{v}~d\mu,$$
what allows to apply Lemma~\ref{lemma:char_solution}. The uniqueness of the solution follows from Lemma~\ref{lemma:uniqueness}.  
\end{proof}

\appendix
\section{Feller transition functions}\label{appendix:revuz}
Let $E$ be a locally compact separable metric space. Denote by $C_c(E)$ the space of compactly supported continuous real valued functions and by $C_0(E)$ the closure of $C_c(E)$ in $(C(E),\|\cdot\|_\infty)$.

A function $K:E\times\mathcal{B}(E)\to[0,\infty)$ is called \emph{Markovian kernel} if for every $x\in E$ the map $K(x,\cdot)$ is a nonnegative measure on $\mathcal{B}(E)$ of total mass $K(x,E)\leq 1$ and for every $A\in \mathcal{B}(E)$ the map $K(\cdot, A)$ is measurable. Given a Markovian kernel $K$ and a function $u$ we write $Ku (x):=\int_E u(y) K(x,dy)$ whenever the integral exists. We call a family $(P_t)_{t>0}$ of Markovian kernels $P_t$ a \emph{Feller transition function} on $E$ if $P_t C_0(E)\subseteq C_0(E)$, for every $s,t>0$ and every bounded measurable $u$ one has $P_tP_su(x)=P_{t+s} u(x)$ for every $x\in E$ and $\lim_{t\to 0}\left\|P_t u- u\right\|_\infty=0$ for every $u\in C_0(E)$.

A Feller transition function  $(P_t)_{t>0}$ is said to have the \emph{strong Feller property} if for any $t>0$ and any bounded Borel function $v$ the function $P_tv$ is bounded and continuous. 

It is known, c.f. \cite[Chapter I, (9.4) Theorem]{BG68} or \cite[Theorem A.2.2]{FOT}, that for every Feller transition function $(P_t)_{t>0}$ on $E$ there is a Hunt process $Y=(Y_t)$ with state space $E$ in correspondence with $(P_t)_{t>0}$, i.e. such that $P_t(x,E)=\IP_x(Y_t\in E)$ holds for every $x\in X$ and $E\in \mathcal{B}(X)$. Here $\IP_x$ the probability under the condition $Y_0=x$ for a given point $x\in E$. For details see \cite{BG68, FOT}. 

\section{Dirichlet forms, semigroups and heat kernels}\label{appendix:Dirichlet}

Suppose $m$ is a nonnegative Radon measure $E$ with full support. A symmetric nonnegative definite bilinear form $Q:D(Q)\times D(Q)\to \mathbb{R}$, defined on a dense subspace $D(Q)$ of $L^2(E,m)$ is called a \emph{closed form} if 
$D(Q)$, equipped with the scalar product $\langle\cdot,\cdot\rangle_Q:=Q(\cdot,\cdot)+\langle \cdot,\cdot\rangle_{L^2(E,m)}$, is a Hilbert space. A closed form $Q$ is called a \emph{Dirichlet form} if in addition for every $f\in D(Q)$ we have $0\vee f\wedge 1\in D(Q)$ and $Q(0\vee f\wedge 1)\leq Q(f)$. A Dirichlet form $Q$ is called \emph{regular} if $C_c(E)\cap D(Q)$ is both uniformly dense in $C_c(E)$ and $\|\cdot\|_Q$-dense in $D(Q)$. See \cite[Chapter 1]{FOT}.

For each closed form $Q$ there is a unique nonnegative definite self-adjoint operator $L$ in correspondence with $Q$, defined by 
\begin{align}\label{E:saop} D(L)=\{u\in D(Q)\colon& \text{there is } f\in L^2(E,m) \text{ such that }\notag\\
&Q(u,v)=\langle f,v\rangle_{L^2(E,m)} \text{ holds for all } v\in D(Q)\},
\end{align}
\[L u=f,\]
and called the \emph{generator of $Q$}. If $Q$ is a Dirichlet form then $(e^{-tL})_{t>0}$ is a strongly continuous semigroup of positivity preserving and contractive self-adjoint operators $e^{-tL}$ on $L^2(E,m)$. By $((L+\alpha)^{-1})_{\alpha>0}$ we denote the corresponding strongly continuous resolvent. 

We say that a Feller transition function $(P_t)_{t>0}$ on $E$ \emph{is in correspondence with a Dirichlet form $Q$} if for every $u\in L^2(E,m)$ and for every representative $v$ of $u$ and for every $t>0$ the function $P_t v$ is a representative of $e^{-tL} u$. If such a Feller semigroup exists, then $Q$ is regular.

If a Dirichlet form $Q$ is regular, then there exists a Hunt process $Y=(Y_t)_{t\geq 0}$ on $E$, unique in a suitable sense, such that for all Borel sets $A\subset E$ of finite measure and for $m$-a.e. $x\in A$ we have $e^{-tL}\mathbf{1}_A(x)=\IP_x(Y_t\in A)$.

Suppose $Q$ is a Dirichlet form with generator $L$. A family $\left\lbrace p_t(\cdot,\cdot)\right\rbrace_{t>0}$ of measurable functions $p_t:E\times E\to [0,+\infty)$ is called a \emph{heat kernel for $(e^{-Lt})_{t>0}$} if for every $t>0$, every $f\in L^2(E,m)$ and $m$-a.e. $x\in X$ we have 
\[e^{-tL}f(x)=\int_E f(y) p_t(x,y) dm(y).\] 
If $\left\lbrace p_t(\cdot,\cdot)\right\rbrace_{t>0}$ is a heat kernel for $(e^{-Lt})_{t>0}$ then we have $p_t(x,y)=p_t(y,x)$ and 
\begin{equation}\label{E:CK}
p_{t+s}(x,y)=\int_E p_t\left(x,z\right)p_s\left(z,y\right)dm(z).
\end{equation}
for all $s,t>0$ and $m$-a.e. $x,y\in E$.

The semigroup $(e^{-tL})_{t\geq 0}$ is called  \emph{ultracontractive} if there exists a decreasing function $\gamma:(0,+\infty)\to (0,+\infty)$ such that for every $t>0$ and $u\in L^2(E,m)$ we have 
\begin{equation}\label{E:ratefct}
\left\|e^{-tL}u\right\|_{L^\infty(E,m)}\leq \gamma(t)\left\|u\right\|_{L^2(E,m)}.
\end{equation}
To the function $\gamma$ we refer as \emph{rate function}, \cite[Chapter 14]{Grig09}. See \cite{CKS87, D89, C96} for further information on ultracontractive semigroups.

If $(e^{-tL})_{t>0}$ is ultracontractive with rate function $\gamma$ and $\left\lbrace p_t(\cdot,\cdot)\right\rbrace_{t>0}$ is a heat kernel for $(e^{-tL})_{t>0}$ then we have $p_{2t}(x,y)\leq \gamma(t)^2$
for every $t>0$ and $m$-a.e. $x,y\in E$, see \cite[Lemma 2.1.2.]{D89}. 

If $m(E)<+\infty$ and $(e^{-tL})_{t>0}$ is ultracontractive, then the generator $L$ has purely discrete spectrum, see for instance \cite[Theorem 2.1.4.]{D89}. In this case the heat kernel shows a typical mixing property, \cite[Theorem 2.4]{bass_hsu}, \cite[Proposition 2.2]{Pardoux99}: For every $t_0>0$ there are constants $c_1,c_2>0$ such that the inequality 
\begin{equation}\label{E:mixing}
\left|p_t(x,y)-\frac{1}{m(X)}\right|\leq c_1e^{-c_2 t}
\end{equation} 
holds for every $t> t_0$ and $m$-a.e. $x,y\in E$. The proof carries over from \cite[Theorem 2.4]{bass_hsu}, it makes use of the eigenfunction expansion of the heat kernel and ultracontractivity.

The setup in Section \ref{section:Preliminaries} arises as the special case where $E=X$ is a countable set, endowed with the discrete topology, and $m$ is a strictly positive function on $X$ so that $(X,m)$ is a discrete measure space. If $b$ is a graph on $X$ and $L$ is the associated Neumann Laplacian, then for any $t\geq 0$ we can define a symmetric function $p_t: X\times X \to (0,+\infty)$ by 
\begin{equation}\label{E:hk}
p_t(x,y)=\frac{1}{m(y)}e^{-tL} 1_{y}(x),
\end{equation}
now for all $x,y\in X$, and it is obvious that $\left\lbrace p_t(\cdot,\cdot)\right\rbrace_{t>0}$ is a heat kernel for the semigroup $(e^{-tL})_{t>0}$. It satisfies the heat equation, more precisely, for any fixed $y\in X$ we have  
\begin{equation}\label{E:heateqforp}
\frac{d}{ds}p(s,\cdot,y) |_{s=t} =-L p(t,\cdot,y).
\end{equation}
The upper bound
\begin{equation}\label{E:easyhke}
p(t,x,y)\leq\frac{1}{m(x)}, \quad x,y\in X,
\end{equation}
is obvious from (\ref{E:hk}). If $m(X)<+\infty$ then we have
\begin{equation}\label{E:stochcompl}
\sum_{y\in X} p(t,x,y)m(y)=e^{-s L} 1(x)=1
\end{equation}
for any $t>0$ and $x\in X$.

\section{The Royden compactification}\label{appendix:royden}
Let $X$ be a countably infinite set and $b$ a canonically compactifiable graph on $X$. We briefly sketch the construction of the Royden compactification of the graph and provide some basic facts. For more details we refer to \cite{Soardi}.

An easy calculation shows $\widetilde{Q}(fg)\leq \|f\|_\infty\widetilde{Q}(g)+\|g\|_\infty\widetilde{Q}(f)$ for every $f,g\in\widetilde{D}\subseteq \ell^\infty(X)$, what implies that $\widetilde{D}$ is an algebra with respect to pointwise multiplication. The completion 
$$\mathcal{A}:=\overline{\widetilde{D}}^{\|\cdot\|_\infty}$$ 
of $\widetilde{D}$ with respect to $\|\cdot\|_\infty$ is a 	unitary Banach algebra. Applying Gelfand theory to (the complexification of) $\A$ we infer the existence of a unique (up to homeomorphism) separable, compact Hausdorff space $R$ such that $X$ is a dense open subset of $R$, the subspace topology of $X$ in $R$ is the discrete topology,
every function in $\widetilde{D}$ can be uniquely extended to a continuous function on $R$ and the algebra $\widetilde{D}$ separates the points of $R$. See for instance \cite{KLSW17, Soardi} and also \cite{HKT, H, Maeda64}.
The space $R$ is called the {\em Royden compactification} of the graph $b$ over $X$. The  Stone-Weierstra{\ss}  theorem implies that the Banach algebra $\A$ is isomorphic to the algebra $C(R)$ of real-valued continuous functions on $R$. The set $\partial_R X:=R\setminus X$ is called the {\em Royden boundary} of the graph. For canonically compactifiable graphs an elementary proof shows that the Royden compactification is metrizable, see the following Lemma \ref{lemma:royden_metrizable}. In general the Royden compactification of a topological space does not have to be metrizable.

 \begin{lemma}\label{lemma:royden_metrizable}
 Let $b$ be a canonically compactifiable graph over $X$. Then, the Royden compactification is metrizable.
 \end{lemma}
  
 \begin{proof}
 We need to show that the algebra $\mathcal{A}$ is separable with respect to $\|\cdot\|_\infty$. To do so it suffices to show that $\widetilde{D}$ is separable with respect to 
 $\|\cdot\|_\infty$. Let $m$ be an arbitrary measure on $X$ such that $(X,m)$ is a finite discrete measure space. Using Lemma~\ref{lemma:embedding_Neumann_supnorm} it suffices to find a countable subfamily of $\widetilde{D}$ that is dense in $\widetilde{D}(=D(Q))$ with respect to $\|\cdot\|_{Q}$. Since $L$ has discrete spectrum $0\leq\lambda_1\leq\lambda_2\leq\ldots$ and since $\ell^2(X,m)$ is separable, there is an orthonormal basis of eigenfunctions $(\psi_n)_n$ of $\ell^2(X,m)$. 
 Then, a simple calculation shows that $(\psi_n)_n$ is an orthogonal system in $(D(Q),\|\cdot\|_Q)$. Suppose $f\in D(Q)$, $f\not=0$, satisfies $\langle f,\psi_n\rangle_Q=0$ for every $n$. 
 Then, we infer $$0=\langle f,\psi_n\rangle_Q=(1+\lambda_n)\langle f,\psi_n\rangle_{\ell^2(X,m)}$$ and, hence, $\langle f,\psi_n\rangle_{\ell^2(X,m)}=0$ for every $n\in \IN$. This is a contradiction, since the $(\psi_n)$ are a basis of $\ell^2(X,m)$. Finally, note that $\operatorname{span}\{\psi_1,\psi_2,\ldots\}$ has a $\|\cdot\|_Q$-dense countable subset, given by finite linear combinations $\sum_{i=1}^m \alpha_i\psi_i$ with rational coefficients $\alpha_1,\ldots,\alpha_n$.
 
 \end{proof}
 
%
%

\bibliography{Literature}{}
\bibliographystyle{alpha}
\end{document}